\documentclass[12pt]{article}

\usepackage{amsmath, epsfig, cite}
\usepackage{amssymb}
\usepackage{amsfonts}
\usepackage{latexsym}
\usepackage{amsthm}

\newtheorem{thm}{Theorem}[section]
\newtheorem{prop}[thm]{Proposition}

\newtheorem{cor}[thm]{Corollary}

\newtheorem{lem}[thm]{Lemma}

\numberwithin{equation}{section}

\renewcommand{\thefootnote}{}

\begin{document}

\begin{center}
{\large\bf New $q$-supercongruences arising \\from a summation of
basic hypergeometric series

 \footnote{ The corresponding author$^*$. Email addresses: weichuanan78@163.com (C. Wei), lichun@hainnu.\\edu.cn (C. Li).}}
\end{center}

\renewcommand{\thefootnote}{$\dagger$}

\vskip 2mm \centerline{$^{1}$Chuanan Wei, $^{2}$Chun Li $^{*}$}
\begin{center}
{$^{1}$School of Biomedical Information and Engineering,\\
  Hainan Medical University, Haikou 571199, China\\
  $^{2}$Key Laboratory of Data Science and Intelligence Education
 of\\ Ministry of Education, Hainan Normal University, Haikou 571158, China}
\end{center}


\vskip 0.7cm \noindent{\bf Abstract.} With the help of  a summation
of basic hypergeometric series,  the creative microscoping method
recently introduced by Guo and Zudilin, and the Chinese remainder
theorem for coprime polynomials, we find some new
$q$-supercongruences. Especially, we give a
$q$-analogue of a formula due to Liu [J. Math. Anal. Appl. 497
(2021), Art.~124915].



\section{Introduction}
For any complex variable $x$, define the shifted-factorial to be
\[(x)_{0}=1\quad \text{and}\quad (x)_{n}
=x(x+1)\cdots(x+n-1)\quad \text{when}\quad n\in\mathbb{Z}^{+}.\] Let
$p$ be an odd prime and let $\mathbb{Z}_p$ denote the ring of all
$p$-adic integers. Define Morita's $p$-adic Gamma function (cf.
\cite[Chapter 7]{Robert}) by
 \[\Gamma_{p}(0)=1\quad \text{and}\quad \Gamma_{p}(n)
=(-1)^n\prod_{\substack{1\leqslant k< n\\
p\nmid k}}k,\quad \text{when}\quad n\in\mathbb{Z}^{+}.\] Noting
$\mathbb{N}$ is a dense subset of $\mathbb{Z}_p$ related to the
$p$-adic norm $|\cdot|_p$, for each $x\in\mathbb{Z}_p$, the
definition of $p$-adic Gamma function can be extended as
 \[\Gamma_{p}(x)
=\lim_{\substack{n\in\mathbb{N}\\
|x-n|_p\to0}}\Gamma_{p}(n).\]
 Two properties of the $p$-adic Gamma function in common use can be stated as follows:

  \begin{equation*}
\frac{\Gamma_p(x+1)}{\Gamma_p(x)}=
\begin{cases}
 -x, &\text{if $p\nmid x$,}\\[10pt]
-1, &\text{if $p\,|\,x$,}
 \end{cases}
 \end{equation*}
  \begin{equation*}
\Gamma_p(x)\Gamma_p(1-x)=(-1)^{\langle-x\rangle_p-1},
 \end{equation*}
where $\langle x\rangle_p$ indicates the least nonnegative residue of $x$ modulo $p$, i.e., $\langle x\rangle_p\equiv x\pmod p$
and $\langle x\rangle_p\in\{0,1,\ldots,p-1\}$.
  In 2016, Long and Ramakrishna
\cite[Proposition 25]{LR} showed that, for any prime $p$,
\begin{equation}\label{eq:long}
\sum_{k=0}^{p-1}\frac{(1/3)_k^3}{k!^3}\equiv
\begin{cases} \displaystyle \Gamma_p(1/3)^6  \pmod{p^3}, &\text{if $p\equiv 1\pmod 6$,}\\[10pt]
\displaystyle -\frac{p^2}{3}\Gamma_p(1/3)^6\pmod{p^3}, &\text{if
$p\equiv 5\pmod 6$.}
\end{cases}
\end{equation}
Similarly, Liu \cite[Theorem 1.1]{LJ} proved that, for any prime
$p$,
\begin{equation}\label{eq:liu}
\sum_{k=0}^{p-1}\frac{(-1/3)_k^3}{k!^3}\equiv
\begin{cases} \displaystyle -18p^2\Gamma_p(2/3)^6  \pmod{p^3}, &\text{if $p\equiv 1\pmod 6$,}\\[10pt]
\displaystyle 54\Gamma_p(2/3)^6\pmod{p^3}, &\text{if $p\equiv 5\pmod
6$.}
\end{cases}
\end{equation}

 For any complex numbers $x$ and $q$, define the $q$-shifted factorial
 to be
 \begin{equation*}
(x;q)_{0}=1\quad\text{and}\quad
(x;q)_n=(1-x)(1-xq)\cdots(1-xq^{n-1})\quad \text{when}\quad
n\in\mathbb{Z}^{+}.
 \end{equation*}
For simplicity, we also adopt the compact notation
\begin{equation*}
(x_1,x_2,\dots,x_m;q)_{n}=(x_1;q)_{n}(x_2;q)_{n}\cdots(x_m;q)_{n}.
 \end{equation*}
Following Gasper and Rahman \cite{Gasper}, define the basic
hypergeometric series $_{r+1}\phi_{r}$ by
$$
_{r+1}\phi_{r}\left[\begin{array}{c}
a_1,a_2,\ldots,a_{r+1}\\
b_1,b_2,\ldots,b_{r}
\end{array};q,\, z
\right] =\sum_{k=0}^{\infty}\frac{(a_1,a_2,\ldots, a_{r+1};q)_k}
{(q,b_1,b_2,\ldots,b_{r};q)_k}z^k.
$$

Recently, Guo \cite{Guo-adb} established three $q$-supercongruences
 via the creative microscoping method (introduced by Guo and Zudilin \cite{GuoZu}), and the Chinese
remainder theorem for polynomials. Similarly, Wei, Liu, and
Wang\cite[Theorems 1.1 and 1.2]{Wei-b} provided a $q$-analogue of
\eqref{eq:long}. For more $q$-analogues of supercongruences, we
refer the reader to
\cite{Guo-diff,Guo-ijnt,GS20c,GS20d,GuoZu-b,GS,Li,LW,LP,Tauraso,WY,WY-a,WY-b,WY-c,Wei-a,
Zu19}.

Let $[n]=(1-q^n)/(1-q)$ be the $q$-integer and $\Phi_n(q)$ the $n$-th cyclotomic polynomial in $q$:
\begin{equation*}
\Phi_n(q)=\prod_{\substack{1\leqslant k\leqslant n\\
\gcd(k,n)=1}}(q-\zeta^k),
\end{equation*}
where $\zeta$ is an $n$-th primitive root of unity. Motivated by the
work just mentioned, we shall establish the following two theorems.

\begin{thm}\label{thm-a}
Let $n$ be a positive integer with $n\equiv 1\pmod 6$. Then, modulo
$\Phi_n(q)^3$,
\begin{align*}
\sum_{k=0}^{(2n+1)/3}\frac{(q^{-1};q^3)_k^3}{(q^3;q^3)_k^3}q^{9k}
&\equiv
q^{(2-2n)/3}(1+q)\frac{(q;q^3)_{(2n+1)/3}^2}{(q^3;q^3)_{(2n+1)/3}^2}\\[5pt]
&\quad\times\bigg\{3-[2n]^2\bigg(\sum_{i=1}^{(2n+1)/3}\frac{3q^{3i-2}}{[3i-2]^2}-\frac{1+5q+3q^2}{1+q}\bigg)\bigg\}.
\end{align*}
\end{thm}

\begin{thm}\label{thm-b}
Let $n$ be a positive integer with $n\equiv 5\pmod 6$. Then, modulo
$\Phi_n(q)^3$,
\begin{align*}
\sum_{k=0}^{(n+1)/3}\frac{(q^{-1};q^3)_k^3}{(q^3;q^3)_k^3}q^{9k}
&\equiv
q^{(2-n)/3}(1+q)\frac{(q;q^3)_{(n+1)/3}^2}{(q^3;q^3)_{(n+1)/3}^2}\\[5pt]
&\quad\times\bigg\{\theta_n(q)+[n]^2\bigg(\sum_{i=1}^{(n+1)/3}\frac{3q^{3i-2}}{[3i-2]^2}-\frac{1+5q+3q^2}{1+q}\bigg)\bigg\},
\end{align*}
where
$$\theta_n(q)=\frac{(1-q-3q^2)(1-2q^n)+(4-4q-6q^2+3q^3)q^{2n}}{(1+q)(q-q^n)^2}.$$
\end{thm}

It is not difficult to understand that Theorems \ref{thm-a} and
\ref{thm-b} give a $q$-analogue of \eqref{eq:liu}. Letting $n=p$ be
an prime and taking $q\to 1$ in the above two theorems, we obtain
the following conclusions.

\begin{cor}\label{cor-a}
Let $p$ be an prime such that $p\equiv 1\pmod{6}$. Then
\begin{align*}
\sum_{k=0}^{(2p+1)/3}\frac{(-1/3)_k^3}{k!^3} \equiv
\frac{6(1/3)_{(2p+1)/3}^2}{(1)_{(2p+1)/3}^2}\bigg\{1+6p^2-\sum_{i=1}^{(2p+1)/3}\frac{4p^2}{(3i-2)^2}\bigg\}\pmod{p^3}.
\end{align*}
\end{cor}

\begin{cor}\label{cor-b}
Let $p$ be an prime such that $p\equiv 5\pmod{6}$. Then
\begin{align*}
\sum_{k=0}^{(p+1)/3}\frac{(-1/3)_k^3}{k!^3} \equiv
\frac{54(1/3)_{(p+1)/3}^2}{(1)_{(p-2)/3}^2}\bigg\{1+\frac{p^2}{(p+1)^2}\sum_{i=1}^{(p+1)/3}\frac{1}{(3i-2)^2}\bigg\}\pmod{p^3}.
\end{align*}
\end{cor}

In order to explain the equivalence of \eqref{eq:liu}
 and Corollaries \ref{cor-a} and \ref{cor-b}, we need to verify the following relations.

\begin{prop}\label{prop-a}
Let $p$ be a prime such that $p\equiv 1\pmod{6}$. Then
\begin{align*}
\frac{(1/3)_{(2p+1)/3}^2}{(1)_{(2p+1)/3}^2}\bigg\{1+6p^2-\sum_{i=1}^{(2p+1)/3}\frac{4p^2}{(3i-2)^2}\bigg\}
\equiv-3p^2\Gamma_p(2/3)^6 \pmod{p^3}.
\end{align*}
\end{prop}

\begin{prop}\label{prop-b}
Let $p$ be a prime such that $p\equiv 5\pmod{6}$. Then
\begin{align*}
\frac{(1/3)_{(p+1)/3}^2}{(1)_{(p-2)/3}^2}\bigg\{1+\frac{p^2}{(p+1)^2}\sum_{i=1}^{(p+1)/3}\frac{1}{(3i-2)^2}\bigg\}
\equiv\Gamma_p(2/3)^6\pmod{p^3}.
\end{align*}
\end{prop}

The rest of the paper is arranged as follows. The proof of Theorems \ref{thm-a} and \ref{thm-b} will be given in Section 2. To this end, we first derive
a $q$-supercongruence modulo $(1-aq^{tn})(a-q^{tn})(b-q^{tn})$, where
$t\in\{1,2\}$, by using a summation of basic hypergeometric series,  the creative microscoping
method, and the Chinese remainder theorem for coprime polynomials.  Finally, the proof of Propositions \ref{prop-a} and
\ref{prop-b} will be displayed in Section 3.

\section{Proof of Theorems \ref{thm-a} and \ref{thm-b}}
In order to prove Theorems \ref{thm-a} and \ref{thm-b}, we require the
following lemma.

\begin{lem}\label{lem-a}
\begin{align*}
& _{3}\phi_{2}\!\left[\begin{array}{cccccccc}
  a, b, q^{-m}\\
  q,  abq^{2-m}
\end{array};q,\, q^3 \right]
=\frac{(1/a, 1/b;q)_{m}}{(q,
1/ab;q)_{m}}\\[5pt]
&\quad\times\bigg\{\frac{q^m(1-q^m)(q-abq^2-(1+q-aq-bq)q^m)}{(1-abq)(aq-q^m)(bq-q^m)}-\frac{1-ab-(2-a-b)q^m}{(1-a)(1-b)}\bigg\}.
\end{align*}
\end{lem}

\begin{proof}
By comparing the $k$-th summands in the summations, it is easy to
see that
\begin{align*}
&{_{4}\phi_{3}}\!\left[\begin{array}{c}
a,\, b,\, xq,\, q^{-m} \\
cq,\, x,\, abq^{1-m}/c
\end{array};q,\,q  \right]\\[5pt]
&\quad=\frac{(1-c)(ab-cxq^m)}{(1-x)(ab-c^2q^m)}
{_{3}\phi_{2}}\!\left[\begin{array}{c}
a,\, b,\, q^{-m} \\
c,\, abq^{1-m}/c
\end{array};q,\,q  \right]\\[5pt]
&\quad\quad+\frac{(c-x)(ab-cq^m)}{(1-x)(ab-c^2q^m)}
{_{3}\phi_{2}}\!\left[\begin{array}{c}
a,\, b,\, q^{-m} \\
cq,\, abq^{-m}/c
\end{array};q,\,q  \right].
\end{align*}
Evaluating the two series on the right-hand side by
 $q$-Saalsch\"{u}tz identity (cf. \cite[Appendix
(II.12)]{Gasper}):
\begin{align*}
_{3}\phi_{2}\!\left[\begin{array}{c}
a,\, b,\, q^{-m} \\
c,\, abq^{1-m}/c
\end{array};q,\,q  \right]
=\frac{(c/a,c/b;q)_m}{(c,c/ab;q)_m},
\end{align*}
we get
\begin{align}
_{4}\phi_{3}\!\left[\begin{array}{c}
a,\, b,\, xq,\, q^{-m} \\
cq,\, x,\, abq^{1-m}/c
\end{array};q,\,q  \right]=\Omega_m(q;a,b,c,x), \label{saal-b}
\end{align}
where
\begin{align*}
\Omega_m(q;a,b,c,x)&=\frac{(c/a,c/b;q)_m}{(qc,c/ab;q)_m}\\
&\quad\times\left\{\frac{(1-cq^m)(ab-cxq^{m})}{(1-x)(ab-c^2q^{m})}
+\frac{(c-x)(ab-c)(a-cq^{m})(b-cq^{m})}{(1-x)(a-c)(b-c)(ab-c^2q^{m})}\right\}.
\end{align*}
Similarly, it is also routine to confirm the relation
\begin{align*}
&{_{5}\phi_{4}}\!\left[\begin{array}{c}
a,\, b,\, xq,\, yq,\,q^{-m}\\
cq^2,\, x,\, y,\, abq^{1-m}/c
\end{array};q,\,q  \right]\\[5pt]
&\quad=\frac{(1-cq)(ab-cyq^m)}{(1-y)(ab-c^2q^{m+1})}
{_{4}\phi_{3}}\!\left[\begin{array}{c}
a,\, b,\, xq,\, q^{-m} \\
cq,\, x,\, abq^{1-m}/c
\end{array};q,\,q  \right]\\[5pt]
&\quad\quad+\frac{(cq-y)(ab-cq^m)}{(1-y)(ab-c^2q^{m+1})}
{_{4}\phi_{3}}\!\left[\begin{array}{c}
a,\, b,\, xq,\, q^{-m} \\
cq^2,\, x,\, abq^{-m}/c
\end{array};q,\,q  \right].
\end{align*}
Calculating the two series on the right-hand side via
\eqref{saal-b}, we arrive at
\begin{align*}
&_{5}\phi_{4}\!\left[\begin{array}{c}
a,\, b,\, xq,\, yq,\,q^{-m}\\
cq^2,\, x,\, y,\, abq^{1-m}/c
\end{array};q,\,q  \right]  \notag\\[5pt]
&\quad=\frac{(1-cq)(ab-cyq^{m})}{(1-y)(ab-c^2q^{m+1})}\Omega_m(q;a,b,c,x)
 \notag\\[5pt]
&\qquad+\frac{(cq-y)(ab-cq^{m})}{(1-y)(ab-c^2q^{m+1})}\Omega_m(q;a,b,cq,x).
\end{align*}
Letting $c\to q^{-1}, \,x\to\infty,\, y\to\infty$ in the last
equation, we are led to Lemma \ref{lem-a}.
\end{proof}

Subsequently, we shall deduce the following united parametric
extension of Theorems \ref{thm-a} and \ref{thm-b}.

\begin{thm}\label{thm-c}
Let $n$ be a positive integer with $n\equiv 3-t\pmod 3$ and
$t\in\{1,2\}$. Then, modulo $(1-aq^{tn})(a-q^{tn})(b-q^{tn})$,
\begin{align}
&\sum_{k=0}^{(tn+1)/3}\frac{(aq^{-1},q^{-1}/a,q^{-1}/b;q^3)_k}{(q^3;q^3)_k^2(q^3/b;q^3)_k}q^{9k}
\notag\\[5pt]
&\quad\equiv\,\frac{(b-q^{tn})(ab-1-a^2+aq^{tn})}{(a-b)(1-ab)}\frac{(bq,q;q^3)_{(tn+1)/3}}{(bq)^{(tn+1)/3}(1/b,q^3;q^3)_{(tn+1)/3}}A_n(q;b,t)
\notag\\[5pt]
&\qquad+\:
\frac{(1-aq^{tn})(a-q^{tn})}{(a-b)(1-ab)}\frac{(aq,q/a;q^3)_{(tn+1)/3}}{b^{(tn+1)/3}(1/b,1/bq;q^3)_{(tn+1)/3}}B(q;a,b),\label{eq:wei-aa}
\end{align}
where
\begin{align*}
A_n(q;b,t)&=\frac{b(1-q^{tn+1})\{q^{tn+2}/b-q+q^{tn-1}(1+q^3-q^{tn+2}-q^2/b)\}}{(1-q)(1-bq^{tn-1})(1-q^{tn+1}/b)}\\
&\quad-\frac{1-q^{tn-2}/b-q^{tn+1}(2-q^{tn-1}-q^{-1}/b)}{(1-q^{tn-1})(1-q^{-1}/b)},\\
B(q;a,b)&=\frac{(1-bq)\{1-q-b(q^{-2}+q-a-1/a)\}}{q(1-q)(1-ab/q)(1-b/aq)}\\
&\quad-\frac{1-q^{-2}-b(2q-a-1/a)}{bq(1-aq^{-1})(1-q^{-1}/a)}.
\end{align*}
\end{thm}

\begin{proof}
When $a=q^{-tn}$ or $a=q^{tn}$, the left-hand side of
\eqref{eq:wei-aa} is equal to
\begin{align}
\sum_{k=0}^{(tn+1)/3}\frac{(q^{-1-tn},q^{-1+tn},q^{-1}/b;q^3)_k}{(q^3;q^3)_k^2(q^{3}/b;q^3)_k}q^{9k}
= {_{3}\phi_{2}}\!\left[\begin{array}{cccccccc}
 q^{-1-tn},  q^{-1+tn}, q^{-1}/b\\
 q^3,q^{3}/b
\end{array};q^3,\, q^9 \right].
 \label{eq:saal-aa}
\end{align}
According to Lemma \ref{lem-a}, the right-hand side of
\eqref{eq:saal-aa} can be written as
\begin{align*}
\frac{(bq,q;q^3)_{(tn+1)/3}}{(bq)^{(tn+1)/3}(1/b,q^3;q^3)_{(tn+1)/3}}A_n(q;b,t).
\end{align*}
Since $(1-aq^{tn})$ and $(a-q^{tn})$ are relatively prime
polynomials, we have the following result: modulo
$(1-aq^{tn})(a-q^{tn})$,
\begin{align}
\sum_{k=0}^{(tn+1)/3}\frac{(aq^{-1},q^{-1}/a,q^{-1}/b;q^3)_k}{(q^3;q^3)_k^2(q^3/b;q^3)_k}q^{9k}
\equiv\frac{(bq,q;q^3)_{(tn+1)/3}}{(bq)^{(tn+1)/3}(1/b,q^3;q^3)_{(tn+1)/3}}A_n(q;b,t).
\label{eq:wei-bb}
\end{align}

When $b=q^{tn}$, the left-hand side of  \eqref{eq:wei-aa} is equal
to
\begin{align}
\sum_{k=0}^{(tn+1)/3}\frac{(aq^{-1},q^{-1}/a,q^{-1-tn};q^3)_k}{(q^3;q^3)_k^2(q^{3-tn};q^3)_k}q^{9k}
= {_{3}\phi_{2}}\!\left[\begin{array}{cccccccc}
 aq^{-1},  q^{-1}/a, q^{-1-tn}\\
 q^3,q^{3-tn}
\end{array};q^3,\, q^9 \right].
 \label{eq:saal-bb}
\end{align}
By Lemma \ref{lem-a},  the right-hand side of \eqref{eq:saal-bb}
can be expressed as
\begin{align*}
&\frac{(aq,q/a;q^3)_{(tn+1)/3}}{(q^2,q^3;q^3)_{(tn+1)/3}}\\
&\times\bigg\{\frac{q^{tn}(1-q^{tn+1})\{1-q-q^{tn}(q^{-2}+q-a-1/a)\}}{(1-q)(1-aq^{tn-1})(1-q^{tn-1}/a)}
-\frac{1-q^{-2}-q^{tn}(2q-a-1/a)}{(1-aq^{-1})(1-q^{-1}/a)}\bigg\}.
\end{align*}
Then we obtain the conclusion: modulo $(b-q^{tn})$,
\begin{align}
\sum_{k=0}^{(tn+1)/3}\frac{(aq^{-1},q^{-1}/a,q^{-1}/b;q^3)_k}{(q^3;q^3)_k^2(q^3/b;q^3)_k}q^{9k}
\equiv\frac{(aq,q/a;q^3)_{(tn+1)/3}}{b^{(tn+1)/3}(1/b,1/bq;q^3)_{(tn+1)/3}}B(q;a,b).
\label{eq:wei-cc}
\end{align}

It is clear that the polynomials $(1-aq^{tn})(a-q^{tn})$ and
$(b-q^{tn})$ are relatively prime. Noting the $q$-congruences
\begin{align*}
&\frac{(b-q^{tn})(ab-1-a^2+aq^{tn})}{(a-b)(1-ab)}\equiv1\pmod{(1-aq^{tn})(a-q^{tn})},
\\[5pt]
&\qquad\qquad\frac{(1-aq^{tn})(a-q^{tn})}{(a-b)(1-ab)}\equiv1\pmod{(b-q^{tn})}
\end{align*}
and employing the Chinese remainder theorem for coprime polynomials,
we get Theorem \ref{thm-c} from \eqref{eq:wei-bb} and
\eqref{eq:wei-cc}.
\end{proof}

\begin{proof}[Proof of Theorem \ref{thm-a}]
Letting $b\to1, t=2$ in Theorem \ref{thm-c}, we arrive at the
formula: modulo $\Phi_n(q)(1-aq^{2n})(a-q^{2n})$,
\begin{align}
&\sum_{k=0}^{(2n+1)/3}\frac{(aq^{-1},q^{-1}/a,q^{-1};q^3)_k}{(q^3;q^3)_k^3}q^{9k}
\notag\\[5pt]
&\quad\equiv\,\frac{(1-a)^2+(1-aq^{2n})(a-q^{2n})}{(1-a)^2}
\frac{(q;q^3)_{(2n+1)/3}^2}{q^{(2n+1)/3}(q^3;q^3)_{(2n+1)/3}^2}C_n(q)
\notag\\[5pt]
&\qquad+\:\frac{(1-aq^{2n})(a-q^{2n})}{(1-a)^2}
\frac{(aq,q/a;q^3)_{(2n+1)/3}}{(q^2,q^3;q^3)_{(2n-2)/3}}D(q;a)
\notag\\
&\quad\equiv\frac{(q;q^3)_{(2n+1)/3}^2}{q^{(2n+1)/3}(q^3;q^3)_{(2n+1)/3}^2}C_n(q)+\frac{(1-aq^{2n})(a-q^{2n})}{q^{(2n+1)/3}(1-a)^2}
\notag\\[5pt]
&\qquad\times\bigg\{-\frac{(q;q^3)_{(2n+1)/3}^2}{(q^3;q^3)_{(2n+1)/3}^2}(3q+3q^2)+
\frac{(aq,q/a;q^3)_{(2n+1)/3}}{(q^3;q^3)_{(2n+1)/3}^2}(1-q)^2D(q;a)\bigg\},
\label{eq:wei-dd}
\end{align}
where
\begin{align*}
C_n(q)&=\frac{q^3+q^{2n}(1+q^{4n})(1-3q+q^3-3q^4)}{q(1-q)^2(1-q^{2n-1})^2}\\
&\quad+\frac{q^{4n}(1-3q+6q^2+2q^3-3q^4+3q^5)+q^{8n+3}}{q(1-q)^2(1-q^{2n-1})^2},\\
D(q;a)&=\frac{(1+a+a^2)(a-3aq+q^3+a^2q^3-3aq^4)+3a^2q^2(2+q^3)}{q(1-q)^2(1-aq)^2(1-a/q)^2}.
\end{align*}

 By the L'H\^{o}spital rule, we have
\begin{align*}
&\lim_{a\to1}\frac{(1-aq^{2n})(a-q^{2n})}{(1-a)^2}\Big\{-(q;q^3)_{(2n+1)/3}^2(3q+3q^2)+
(aq,q/a;q^3)_{(2n+1)/3}(1-q)^2D(q;a)\Big\}\\[5pt]
&\quad=-q(1+q)[2n]^2(q;q^3)_{(2n+1)/3}^2\bigg\{\sum_{i=1}^{(2n+1)/3}\frac{3q^{3i-2}}{[3i-2]^2}-\frac{1+5q+3q^2}{1+q}\bigg\}.
\end{align*}

Letting $a\to1$ in \eqref{eq:wei-dd} and utilizing the above limit,
we are led to the $q$-supercongruence: modulo $\Phi_n(q)^3$,
\begin{align*}
&\sum_{k=0}^{(2n+1)/3}\frac{(q^{-1};q^3)_k^3}{(q^3;q^3)_k^3}q^{9k}
\notag\\[5pt]
&\quad\equiv\,
\frac{(q;q^3)_{(2n+1)/3}^2}{q^{(2n+1)/3}(q^3;q^3)_{(2n+1)/3}^2}C_n(q)
\notag\\[5pt]
&\qquad-q(1+q)[2n]^2\frac{(q;q^3)_{(2n+1)/3}^2}{q^{(2n+1)/3}(q^3;q^3)_{(2n+1)/3}^2}\bigg\{\sum_{i=1}^{(2n+1)/3}\frac{3q^{3i-2}}{[3i-2]^2}-\frac{1+5q+3q^2}{1+q}\bigg\}
\notag\\
&\quad\equiv
q^{(2-2n)/3}(1+q)\frac{(q;q^3)_{(2n+1)/3}^2}{(q^3;q^3)_{(2n+1)/3}^2}
\notag\\[5pt]
&\qquad\times\bigg\{3-[2n]^2\bigg(\sum_{i=1}^{(2n+1)/3}\frac{3q^{3i-2}}{[3i-2]^2}-\frac{1+5q+3q^2}{1+q}\bigg)\bigg\}.
\end{align*}
This completes the proof of Theorem \ref{thm-a}.
\end{proof}

\begin{proof}[Proof of Theorem \ref{thm-b}]
Letting $b\to1, t=1$ in Theorem \ref{thm-c}, we obtain the result:
modulo $\Phi_n(q)(1-aq^{n})(a-q^{n})$,
\begin{align}
&\sum_{k=0}^{(n+1)/3}\frac{(aq^{-1},q^{-1}/a,q^{-1};q^3)_k}{(q^3;q^3)_k^3}q^{9k}
\notag\\[5pt]
&\quad\equiv\,\frac{(1-a)^2+(1-aq^{n})(a-q^{n})}{(1-a)^2}
\frac{(q;q^3)_{(n+1)/3}^2}{q^{(n+1)/3}(q^3;q^3)_{(n+1)/3}^2}C_{n/2}(q)
\notag\\[5pt]
&\qquad+\:\frac{(1-aq^{n})(a-q^{n})}{(1-a)^2}
\frac{(aq,q/a;q^3)_{(n+1)/3}}{(q^2,q^3;q^3)_{(n-2)/3}}D(q;a)
\notag\\
&\quad\equiv\frac{(q;q^3)_{(n+1)/3}^2}{q^{(n+1)/3}(q^3;q^3)_{(n+1)/3}^2}C_{n/2}(q)+\frac{(1-aq^{n})(a-q^{n})}{q^{(n+1)/3}(1-a)^2}
\notag\\[5pt]
&\qquad\times\bigg\{\frac{(q;q^3)_{(n+1)/3}^2}{(q^3;q^3)_{(n+1)/3}^2}(3q+3q^2)-
\frac{(aq,q/a;q^3)_{(n+1)/3}}{(q^3;q^3)_{(n+1)/3}^2}(1-q)^2D(q;a)\bigg\}.
\label{eq:wei-ee}
\end{align}

 By the L'H\^{o}spital rule, we have
\begin{align*}
&\lim_{a\to1}\frac{(1-aq^{n})(a-q^{n})}{(1-a)^2}\Big\{(q;q^3)_{(n+1)/3}^2(3q+3q^2)-
(aq,q/a;q^3)_{(n+1)/3}(1-q)^2D(q;a)\Big\}\\[5pt]
&\quad=q(1+q)[n]^2(q;q^3)_{(n+1)/3}^2\bigg\{\sum_{i=1}^{(n+1)/3}\frac{3q^{3i-2}}{[3i-2]^2}-\frac{1+5q+3q^2}{1+q}\bigg\}.
\end{align*}

Letting $a\to1$ in \eqref{eq:wei-ee} and employing the upper limit,
we get the $q$-supercongruence: modulo $\Phi_n(q)^3$,
\begin{align*}
&\sum_{k=0}^{(n+1)/3}\frac{(q^{-1};q^3)_k^3}{(q^3;q^3)_k^3}q^{9k}
\notag\\[5pt]
&\quad\equiv\,
\frac{(q;q^3)_{(n+1)/3}^2}{q^{(n+1)/3}(q^3;q^3)_{(n+1)/3}^2}C_{n/2}(q)
\notag\\[5pt]
&\qquad+q(1+q)[n]^2\frac{(q;q^3)_{(n+1)/3}^2}{q^{(n+1)/3}(q^3;q^3)_{(n+1)/3}^2}\bigg\{\sum_{i=1}^{(n+1)/3}\frac{3q^{3i-2}}{[3i-2]^2}-\frac{1+5q+3q^2}{1+q}\bigg\}
\notag\\
&\quad\equiv
q^{(2-n)/3}(1+q)\frac{(q;q^3)_{(n+1)/3}^2}{(q^3;q^3)_{(n+1)/3}^2}\\[5pt]
&\qquad\times\bigg\{\theta_n(q)+[n]^2\bigg(\sum_{i=1}^{(n+1)/3}\frac{3q^{3i-2}}{[3i-2]^2}-\frac{1+5q+3q^2}{1+q}\bigg)\bigg\}.
\end{align*}
Thus we finish the proof of Theorem \ref{thm-b}.
\end{proof}
\section{Proof of Propositions \ref{prop-a} and \ref{prop-b}}

 Let $\Gamma_p^{'}(x)$ and $\Gamma_p^{''}(x)$
 be the first derivative and second derivative of $\Gamma_p(x)$
 respectively. Now we begin to provide the Proof of Propositions \ref{prop-a} and
 \ref{prop-b}.

\begin{proof}[Proof of Proposition \ref{prop-a}]
By means of the properties of the $p$-adic Gamma function, we arrive
at
\begin{align*}
\frac{(1/3)_{(2p+1)/3}^2}{(1)_{(2p+1)/3}^2}&=\frac{p^2}{(2p+1)^2}\bigg\{\frac{\Gamma_p((2+2p)/3)\Gamma_p(1)}{\Gamma_p(1/3)\Gamma_p((1+2p)/3)}\bigg\}^2
\notag\\[5pt]
&=\frac{p^2}{(2p+1)^2}\big\{\Gamma_p(2/3)\Gamma_p((2+2p)/3)\Gamma_p((2-2p)/3)\big\}^2.
\end{align*}
Moreover, it is not difficult to understand that
\begin{align*}
&1+6p^2-\sum_{i=1}^{(2p+1)/3}\frac{4p^2}{(3i-2)^2}
\notag\\[5pt]
&=-3+6p^2-\sum_{i=1}^{(p-1)/3}\frac{4p^2}{(3i-2)^2}-\sum_{i=(p+5)/3}^{(2p+1)/3}\frac{4p^2}{(3i-2)^2}.
\end{align*}
Then we can proceed as follows:

\begin{align*}
&\frac{(1/3)_{(2p+1)/3}^2}{(1)_{(2p+1)/3}^2}\bigg\{1+6p^2-\sum_{i=1}^{(2p+1)/3}\frac{4p^2}{(3i-2)^2}\bigg\}
\\[5pt]
&\quad=\frac{p^2}{(2p+1)^2}\big\{\Gamma_p(2/3)\Gamma_p((2+2p)/3)\Gamma_p((2-2p)/3)\big\}^2\\[5pt]
&\qquad\times\bigg\{-3+6p^2-\sum_{i=1}^{(p-1)/3}\frac{4p^2}{(3i-2)^2}-\sum_{i=(p+5)/3}^{(2p+1)/3}\frac{4p^2}{(3i-2)^2}\bigg\}
\\[5pt]
&\quad\equiv\frac{-3p^2}{(2p+1)^2}\Gamma_p(2/3)^6
\\[5pt]
&\quad\equiv-3p^2\Gamma_p(2/3)^6\pmod{p^3}.
\end{align*}
This verifies the correctness of Proposition \ref{prop-a}.
\end{proof}

\begin{proof}[Proof of Proposition \ref{prop-b}]
Through the properties of the $p$-adic Gamma function, we have
\begin{align}
\frac{(1/3)_{(p+1)/3}^2}{(1)_{(p-2)/3}^2}&=\bigg\{\frac{\Gamma_p((2+p)/3)\Gamma_p(1)}{\Gamma_p(1/3)\Gamma_p((1+p)/3)}\bigg\}^2
\notag\\[5pt]
&=\big\{\Gamma_p(2/3)\Gamma_p((2+p)/3)\Gamma_p((2-p)/3)\big\}^2
\notag\\[5pt]
&\equiv\Gamma_p(2/3)^2\bigg\{\Gamma_p(2/3)+\Gamma_p^{'}(2/3)\frac{p}{3}+\Gamma_p^{''}(2/3)\frac{p^2}{18}\bigg\}^2
\notag\\[5pt]
&\quad\times\bigg\{\Gamma_p(2/3)-\Gamma_p^{'}(2/3)\frac{p}{3}+\Gamma_p^{''}(2/3)\frac{p^2}{18}\bigg\}^2
\notag\\[5pt]
&\equiv\Gamma_p(2/3)^6\bigg\{1-\frac{2p^2}{9}G_1(2/3)^2+\frac{2p^2}{9}G_2(2/3)\bigg\}\pmod{p^3},
\label{eq:wei-aaa}
\end{align}
where $G_1(x)=\Gamma_p^{'}(x)/\Gamma_p(x)$ and
$G_2(x)=\Gamma_p^{''}(x)/\Gamma_p(x)$.

Let
\[H_{m}
  =\sum_{k=1}^m\frac{1}{k},\quad H_{m}^{(2)}
  =\sum_{k=1}^m\frac{1}{k^{2}}.\]
 In light of the three relations from Wang and Pan \cite[Lemmas 2.3 and 2.4]{Wang}:
\begin{align*}
&G_2(0)=G_1(0)^2
\\[5pt]
 &G_1(2/3)\equiv G_1(0)+{H}_{(2p-1)/3}\pmod{p},
 \\[5pt]
&G_2(2/3)\equiv
G_2(0)+2G_1(0){H}_{(2p-1)/3}+H_{(2p-1)/3}^2-H_{(2p-1)/3}^{(2)}\pmod{p},
\end{align*}
we get
\begin{align}
G_2(2/3)-G_1(2/3)^2\equiv-{H}_{(2p-1)/3}^{(2)}\pmod{p}.
\label{eq:wei-bbb}
\end{align}
In view of \eqref{eq:wei-aaa} and \eqref{eq:wei-bbb}, we are led to
\begin{align}
&\frac{(1/3)_{(p+1)/3}^2}{(1)_{(p-2)/3}^2}\bigg\{1+\frac{p^2}{(p+1)^2}\sum_{i=1}^{(p+1)/3}\frac{1}{(3i-2)^2}\bigg\}\\[5pt]
 &\quad\equiv\Gamma_p(2/3)^6\bigg\{1-\frac{2p^2}{9}{H}_{(2p-1)/3}^{(2)}\bigg\}\bigg\{1+\frac{p^2}{(p+1)^2}\sum_{i=1}^{(p+1)/3}\frac{1}{(3i-2)^2}\bigg\}
\notag\\[5pt]
&\quad\equiv
\Gamma_p(2/3)^6\bigg\{1-\frac{2p^2}{9}{H}_{(2p-1)/3}^{(2)}+\frac{p^2}{(p+1)^2}\sum_{i=1}^{(p+1)/3}\frac{1}{(3i-2)^2}\bigg\}\pmod{p^3}.
 \label{eq:wei-ccc}
\end{align}
It is easy to see that
 \begin{align}
\sum_{i=1}^{(p+1)/3}\frac{1}{(3i-2)^2}&=H_{p-1}^{(2)}-\frac{1}{9}H_{(p-2)/3}^{(2)}-\sum_{i=1}^{(p-2)/3}\frac{1}{(3i-1)^2}
 \notag\\[5pt]
 &\equiv
-\frac{1}{9}H_{(p-2)/3}^{(2)}-\sum_{i=1}^{(p-2)/3}\frac{1}{(3i-1)^2}
\notag\\[5pt]
&=-\frac{1}{9}H_{(p-2)/3}^{(2)}-\sum_{i=1}^{(p-2)/3}\frac{1}{(p-3i)^2}
 \notag\\[5pt]
 &\equiv
-\frac{2}{9}H_{(p-2)/3}^{(2)}\notag\\[5pt]
&=-\frac{2}{9}\sum_{i=(2p+2)/3}^{p-1}\frac{1}{(p-i)^2}
\notag\\[5pt]
 &\equiv
-\frac{2}{9}\sum_{i=(2p+2)/3}^{p-1}\frac{1}{i^2}\notag\\[5pt]
&\equiv\frac{2}{9}H_{(2p-1)/3}^{(2)}
 \pmod{p}. \label{eq:wei-ddd}
\end{align}
 Substituting \eqref{eq:wei-ddd} into  \eqref{eq:wei-ccc}, we confirm the validity of Proposition \ref{prop-b}.
\end{proof}

{\bf{Acknowledgments}}\\

The work is supported by the National Natural Science Foundation of
China (No. 12071103) and the Natural Science Foundation of Hainan
Province (No. 2019RC184).


\end{document}